\documentclass[12pt, reqno, twoside, letterpaper]{amsart}

\usepackage{paperstyle}

\usepackage{graphicx}
\usepackage{float}
\usepackage{mathtools}
\usepackage{bbm}

\makeatletter
\def\section{\@startsection{section}{1}%
  \z@{2.5\linespacing\@plus\linespacing}{.5\linespacing}%
  {\normalfont\scshape\centering}}
\makeatother

\newcommand\chrone{{\chi_\circ}}
\newcommand\ochrone{{\overline\chi_\circ}}
\newcommand\qone{{q_\circ}}
\newcommand\Lone{{\sL_\circ}}

\newcommand{\blltaux}[2]{%
  \vcenter{\hbox{\scalebox{0.9}{$#1\bullet$}}}%
}
\newcommand{\bllt}{\mathpalette\blltaux{}}

\newcommand{\chrtwo}{\chi_{\bllt}}
\newcommand{\ochrtwo}{\overline{\chi}_{\bllt}}
\newcommand{\qtwo}{q_{\bllt}}
\newcommand\Ltwo{{\sL_{\bllt}}}

\title[Unshared zeros of Dirichlet $L$-functions]
{Unshared zeros of Dirichlet $L$-functions}
\author{William Banks}
\address{Department of Mathematics\\
University of Missouri\\
Columbia MO 65211\\
USA}
\email{bankswd@missouri.edu}

\author{Kyle Loftus}
\address{Department of Mathematics\\
University of Missouri\\
Columbia MO 65211\\
USA}
\email{krltc2@missouri.edu}


\date{\today}

\begin{document}

\begin{abstract}
We prove that no Dirichlet $L$-function (and more generally, no
nontrivial finite linear combination of Dirichlet $L$-functions) can
vanish at every zero of a fixed $L(s,\chrone)$. At the heart of the 
proof is a short-window asymptotic for the twisted discrete moment
$\sum_{\rho} x^{\rho} L(\rho,\chrtwo)$, where $\chrone\ne\chrtwo$ are
primitive Dirichlet characters, $\rho=\beta+i\gamma$ runs over zeros
of $L(s,\chrone)$ with $T-\Delta<\gamma\le T$, and $x\in\Z$. The
asymptotic is unconditional, assuming no hypothesis of GRH type, and it
holds for every window width
$\Delta\in[T\,\er^{-C\sqrt{\log T}},\,T/\log T]$,
thus reaching windows shorter than $T(\log T)^{-A}$ for any fixed $A$.
Notably, the main term
$\frac{\chrtwo(x)}{2\pi}\,\Delta\log T$ depends on $x$ only through the
single value $\chrtwo(x)$. Since distinct primitive characters are 
distinguished by their values, varying $x$ isolates the contribution
of each $L$-function within a linear combination, and we deduce that
for a positive density of $x\in\N$, every nontrivial combination is
nonzero at some zero of $L(s,\chrone)$ in any sufficiently high short
window. The proof combines contour integration of
$-\frac{L'}{L}(1-s,\ochrone)\,L(s,\chrtwo)$ with short-interval 
estimates for the Dirichlet convolution $(\chrone\Lambda)*\chrtwo$,
which derive from the classical de la Vall\'ee Poussin zero-free 
region.
\end{abstract}

\maketitle

\tableofcontents

\newpage{\Large\section{Introduction}}

\subsection{Background and motivation}
A recurring theme in analytic number theory is the study of discrete
moments of $L$-functions $L(s,\chi)=\sum\chi(n)n^{-s}$ ($\Re(s)>0$),
such as the sum
\[
\sum_{0<\gamma\le T}L(\rho,\chrtwo),\qquad
\rho=\beta+i\gamma\text{ a zero of }L(s,\chrone),
\]
in which two distinct primitive Dirichlet characters $\chrone$ and
$\chrtwo$ play asymmetric roles. Sums like this measure how one
$L$-function behaves at the zeros of another. Garunk\v stis, Kalpokas,
and Steuding~\cite{GKS} showed that the above sum is $o(T)$ as
$T\to\infty$, which is substantially less than the trivial bound
$\sim\frac{1}{2\pi}T\log(\qone T)$ supplied by the Riemann--von Mangoldt
formula. Garunk\v stis and Kalpokas~\cite{GK} went on to evaluate the
second moment $\sum_\rho|L(\rho,\chrtwo)|^2$, finding a main term of
size $T(\log T)^2$, and Laaksonen and Petridis~\cite{LP} examined
related questions of value distribution.

The standard approach to such sums goes back to Gonek's
evaluation~\cite{Gonek} of $\sum_\rho|\zeta'(\rho)|^2$ and to the work
of Conrey--Ghosh--Gonek~\cite{ConGhoGon} on gaps between zeros. Both
proceed by contour integration of $L'/L$ against a suitable kernel,
combined with the approximate functional equation. Fujii carried this
machinery into the Dirichlet setting in a sequence of
papers~\cite{Fujii_zetazeros, Fujii_zetazerosII, Fujii_observations},
treating sums such as $\sum_{0<\gamma\le T}x^\rho\,\zeta'(\rho)$ and
uncovering a striking arithmetic dichotomy in which the shape of the
main term depends on whether $x$ is rational or irrational.

A second motivation comes from independence phenomena among distinct
Dirichlet $L$-functions. Linnik~\cite{Linnik} observed that the
functional equation of $L(s,\chi)$ can be recovered from that of
$\zeta(s)$, and Sprind\v zuk~\cite{Sprind1, Sprind2} showed that GRH
for all Dirichlet $L$-functions follows from RH together with a
suitable hypothesis on the vertical distribution of zeros of
$\zeta(s)$. The Linnik--Sprind\v zuk phenomenon was developed further
by Fujii~\cite{Fujii_zetazeros, Fujii_zetazerosII,
Fujii_observations}, by Kaczorowski and Perelli~\cite{KP} in the
Selberg class, by Suzuki~\cite{Suzuki}, and recently by the first
author~\cite{Banks, Banks_GRHsingle, Banks_simple}. The common thread
is that fine information about the zeros of a single $L$-function can
be made to control the entire family. Two
questions arise naturally. Can $L(s,\chrtwo)$ vanish at every zero of
$L(s,\chrone)$? More generally, can some finite linear combination
$\sum_j a_j L(s,\chi_j)$ vanish identically on the zero set of
$L(s,\chrone)$?

\smallskip
\subsection{Statement of results}
Our main theorem answers both questions in the negative, and in
quantitative form. The mechanism is a short-window asymptotic for the
twisted moment $\sum_\rho x^\rho L(\rho,\chrtwo)$.

\begin{theorem}\label{thm:BLmain}
There is an absolute constant $C>0$ such that,
for every real $Q\ge 1$ and integer $x\ge 10$, one can
choose $T_0=T_0(Q,x)>0$ with the following property.
Let $\chrone$ and $\chrtwo$ be distinct
primitive characters with moduli $\qone$ and $\qtwo$,
respectively, satisfying $\qone,\qtwo\le Q$.
Then, for all $T\ge T_0$ and
$T\,\er^{-C\sqrt{\log T}}\le\Delta\le T/\log T$, one has
\be\label{eq:BLest}
\ssum{\rho=\beta+i\gamma\\T-\Delta<\gamma\le T}
x^\rho L(\rho,\chrtwo)
=\frac{\chrtwo(x)}{2\pi}\,\Delta\log T
+O(\Delta),
\ee
where the sum runs over the zeros $\rho$ of $L(s,\chrone)$, each
counted with its multiplicity,
and the implied constant depends only on $Q$ and $x$.
\end{theorem}

We comment on three features of Theorem~\ref{thm:BLmain}.
First, the asymptotic holds in short windows. The range
$T\,\er^{-C\sqrt{\log T}}\le\Delta\le T/\log T$ is essentially
the shortest one accessible from the classical zero-free region,
and the lower bound is small enough that $\Delta/T$ can be made
smaller than $(\log T)^{-A}$ for every fixed $A>0$. The upper
bound ensures that the main term, of size $\frac{1}{2\pi}\Delta\log T$,
is comparable to the total number of zeros in the window,
which is roughly $\frac{1}{2\pi}\Delta\log(\qone T)$ by the
Riemann--von Mangoldt formula. The earlier asymptotics
of~\cite{GKS, GK} are stated only over the longer range
$0<\gamma\le T$; we are not aware of any short-window
versions in the literature.

Second, the main term is proportional to the
character value $\chrtwo(x)$, rather than to a
sum or moment built from $\chrtwo$. Since distinct primitive
characters take distinct values, varying $x$ separates different
choices of $\chrtwo$; this feature leads to
Corollary~\ref{cor:BLmain} below. The factor $x^\rho$ in the summand
oscillates as $x^{1/2}\,\er^{i\gamma\log x}$ on the critical
line, and plays the role of an arithmetic twist analogous to the
one introduced by Fujii~\cite{Fujii_zetazeros} for
$\sum_\rho x^\rho\,\zeta'(\rho)$. The novelty here is pairing that 
twist with a different function, so the arithmetic
structure is governed by the Dirichlet convolution
$(\chrone\Lambda)*\chrtwo$ rather than by $\Lambda$ alone.

Third, the theorem is unconditional. We assume no form of GRH,
partial or otherwise, for either $L(s,\chrone)$ or $L(s,\chrtwo)$.
When $x$ is chosen so that $\chrtwo(x)\ne 0$, the error term saves a
factor of $\log T$ over the main term, reflecting the strength
of the de la Vall\'ee Poussin zero-free region. One caveat is that if
$L(s,\chrone)$ or $L(s,\chrtwo)$ admits a Siegel zero, then the
implied constant is ineffective, as is unavoidable in any argument
that invokes Siegel's theorem.

We now turn to the linear-combination consequence promised earlier.

\begin{corollary}\label{cor:BLmain}
Let $\chi_1,\ldots,\chi_k$ be pairwise distinct primitive
characters of moduli at most $Q$, none equal to $\chrone$,
and let $a_1,\ldots,a_k\in\C$ be coefficients that are not all zero. 
Let $\mathbf a$ denote the tuple $(a_1,\ldots,a_k)$. 
Then, with $x$ and $\Delta$ as in Theorem~\ref{thm:BLmain}, there exists
$T_0=T_0(Q,x,k,\mathbf a)>0$ such that for all $T\ge T_0$,
\[
\ssum{\rho=\beta+i\gamma\\T-\Delta<\gamma\le T}
x^\rho\sum_{j=1}^k a_j L(\rho,\chi_j)
=\frac{\psi(x)}{2\pi}\,\Delta\log T
+O(\Delta),
\qquad
\psi(x)\defeq\sum_{j=1}^k a_j\chi_j(x),
\]
where $\rho$ runs over the zeros of $L(s,\chrone)$, each
counted with its multiplicity, and the implied constant
depends only on $Q,x,k,{\mathbf a}$.
In particular, for a positive density of $x\in\N$ and all
$T\ge T_0$, some zero $\rho=\beta+i\gamma$ of $L(s,\chrone)$
with $T-\Delta<\gamma\le T$ satisfies
\[
a_1L(\rho,\chi_1)+\cdots+a_kL(\rho,\chi_k)\ne 0.
\]
\end{corollary}

\begin{proof}
Applying Theorem~\ref{thm:BLmain} with $\chrtwo=\chi_j$ for each
$j=1,\dots,k$, multiplying by $a_j$, and summing over $j$ gives
the stated asymptotic. For the final assertion, it remains to find
a positive density of $x\in\N$ at which the main term survives,
that is, at which $\psi(x)$ is
nonzero. Let $M$ be the least common multiple of the moduli of
$\chi_1,\dots,\chi_k$. By orthogonality,
\[
\sum_{x\bmod M}|\psi(x)|^2=\phi(M)\sum_{j=1}^k|a_j|^2>0,
\]
so $\psi(x)\ne 0$ on at least one reduced residue class modulo $M$,
a set of positive density. For any such $x$, the main term dominates
the error term once $T$ is large enough, so the sum over
$\rho$ is nonzero. Thus, at least one zero $\rho$ satisfies
$a_1L(\rho,\chi_1)+\cdots+a_kL(\rho,\chi_k)\ne 0$.
\end{proof}

\smallskip
Corollary~\ref{cor:BLmain} invites comparison with joint
universality, which for Dirichlet $L$-functions (and, in the form
given by Sander and Steuding~\cite{SanderSteuding}, for their sums
and products) implies that no nontrivial linear combination can
vanish identically on any open subset of the strip
$\tfrac12<\Re(s)<1$. However, the comparison is inexact. Universality
constrains values on open sets, whereas
Corollary~\ref{cor:BLmain} constrains them on the discrete and
arithmetically distinguished set of zeros of a single fixed
$L$-function, a set about which universality says nothing. In this
sense the corollary belongs to the circle of ideas around the
Linnik--Sprind\v zuk theorem rather than to universality.

\smallskip
\subsection{Method and outline}

The proof of Theorem~\ref{thm:BLmain} proceeds by contour
integration. We integrate $x^s$ against the function
\[
K(s)\defeq -\frac{L'}{L}(1-s,\ochrone)\,L(s,\chrtwo)
\]
around a rectangle of height $\Delta$ at vertical
position $T$. The poles of $K$ inside the contour are precisely the
zeros of $L(s,\chrone)$ at which $L(s,\chrtwo)$ does not
vanish. Indeed, at each zero $\rho$ of $L(s,\chrone)$, of multiplicity
$m_\rho$, the factor $-(L'/L)(1-s,\ochrone)$ has a simple pole with
residue $m_\rho$, and so $x^sK(s)$ has residue
$m_\rho\,x^\rho L(\rho,\chrtwo)$ there, which vanishes when
$L(\rho,\chrtwo)=0$.  Summing the residues by Cauchy's theorem
therefore reproduces the left-hand side of \eqref{eq:BLest}, with
each zero counted according to its multiplicity.
On the right edge of the rectangle, the functional equation
\eqref{eq:L7_hist} for the logarithmic derivative exchanges the factor
$\frac{L'}{L}(1-s,\ochrone)$ for $-\frac{L'}{L}(s,\chrone)$ plus an
explicit term of size $\log T$. Paired with the term $n=x$ of the
Dirichlet series of $L(s,\chrtwo)$, the explicit term produces the
main term $\frac{1}{2\pi}\chrtwo(x)\Delta\log T$, while the Dirichlet
series expansion of $\frac{L'}{L}(s,\chrone)\,L(s,\chrtwo)$
contributes only to the error.
On the left edge,
the asymmetric functional equation \eqref{eq:L3_hist} converts
the integral into one whose key input is a bound on the
short-interval sums of the arithmetical function
\[
k(n)\defeq\big\{(\chrone\Lambda)*\chrtwo\big\}(n)
=\ssum{a,b\in\N\\ab=n}\chrone(a)\Lambda(a)\chrtwo(b).
\]
The short-interval bound, established as Lemma~\ref{lem:crucial},
is where the de la Vall\'ee Poussin zero-free region for
$\zeta(s)$ and Dirichlet $L$-functions enters the argument, and
is the source of the gain encoded in the
lower bound $\Delta\ge T\,\er^{-C\sqrt{\log T}}$.

The rest of the paper is organized as follows.
Section~\ref{sec:prelim} assembles the tools of the proof. After
fixing notation in \S\ref{sec:notation}, we recall standard
estimates for Dirichlet $L$-functions in \S\ref{sec:estimates},
establish bounds on several sums in \S\ref{sec:bounds}, chief among
them the short-interval bounds of Lemma~\ref{lem:crucial}, and
record an integral estimate for the factor $\cM_\chi$ of the
functional equation in \S\ref{sec:M}. Section~\ref{sec:BLmain}
proves the main result, Theorem~\ref{thm:BLmain}.

{\Large\section{Preliminaries}\label{sec:prelim}}

\subsection{Notation}\label{sec:notation}

We use the letter $s$ to denote a complex variable, with
$\sigma \defeq \Re s$ and $t \defeq \Im s$.
Throughout the paper we restrict
to $t\ge 10$. Under this standing convention $\log t>0$, and the
estimates below take a cleaner form than in the sources cited while
remaining adequate for our application. This lower bound on $t$
entails no loss of generality. By the reflection principle
$\overline{L(\overline s,\chi)}=L(s,\overline\chi)$,
every estimate below holds equally for $t\le-10$ with $t$ replaced by
$|t|$ and $\chi$ by $\overline\chi$, and the range $|t|<10$ plays no
role in any of our arguments. We therefore state all results for
$t\ge 10$ without further comment, omitting the qualifier from
individual displays unless a different threshold is in force.

Implied constants in the symbols $\ll$ and $O$ are absolute unless
dependence on other parameters is indicated by a subscript, as in
$\ll_Q$ or $O_{Q,x}$. We write $f\asymp g$ to mean $f\ll g\ll f$,
and $\phi$ denotes the Euler totient function.

\subsection{Estimates with $L$-functions}\label{sec:estimates}

Let $\chi$ be a primitive character of modulus $q\ge 1$.
The representation
\be\label{eq:L1_hist}
\frac{L'}{L}(s,\chi)
=\ssum{\rho=\beta+i\gamma\\|\gamma-t|<1}\frac{1}{s-\rho}+O(\log qt)
\qquad (\sigma\in[-1,2])
\ee
is a form of the classical explicit formula. The result originates with
Hadamard~\cite{Hadamard1896}
and de la Vallée Poussin~\cite{VallPoussin1896}
in their proofs of the Prime Number Theorem.
For a modern treatment, we refer the reader to
\cite[Lems.\,12.1, 12.6]{MontVau};
see also \cite[Thms.\,10.13, 10.17]{MontVau}.

The right-hand side of \eqref{eq:L1_hist} becomes unbounded as $s$
approaches a zero of $L(s,\chi)$. Nevertheless, for every $t\ge 10$
there is an ordinate $t^\star\in[t,t+1]$ at which the uniform bound
\be\label{eq:L2_hist}
\frac{L'}{L}(\sigma\pm it^\star,\chi)
\ll (\log qt)^2\qquad (\sigma\in[-1,2])
\ee
holds. This follows from classical zero-density and mean-value
arguments. The idea goes back to Littlewood~\cite{Littlewood1924}
for $\zeta(s)$,
and to classical zero-density methods developed later by
Ingham~\cite{Ingham1932}
and Selberg~\cite{Selberg1946}. For a modern treatment, see
\cite[Lems.\,12.2, 12.7]{MontVau}.

Dirichlet~\cite{Dirichlet1837} introduced the $L$-functions; their
analytic continuation and functional equation were established in
increasing generality by Riemann~\cite{Riemann1859} and
Hurwitz~\cite{Hurwitz1882}.
The asymmetric version of the functional equation
is (cf.~\cite[Cor.\,10.9]{MontVau})
\be\label{eq:L3_hist}
L(s,\chi)=\cM_\chi(s)L(1-s,\overline\chi),
\ee
where $\cM_\chi$ is given by
\[
\cM_\chi(s)\defeq\epsilon\,2^s\pi^{s-1}q^{1/2-s}
\Gamma(1-s)\sin\tfrac{\pi}{2}(s+\kappa)
\]
with
\be\label{eq:constdefns}
\kappa\defeq\begin{cases}
0&\quad\hbox{if $\chi(-1)=+1$,}\\
1&\quad\hbox{if $\chi(-1)=-1$,}\\
\end{cases}
\qquad{\mathfrak g}(\chi)\defeq
\sum_{a\bmod q}\chi(a)\e(a/q),
\qquad\epsilon\defeq\frac{{\mathfrak g}(\chi)}{i^\kappa\sqrt{q}}.
\ee
Here, $\e(u)\defeq\er^{2\pi i u}$ for all $u\in\R$.
The estimate
\be\label{eq:L5_hist}
|L(s,\chi)|
\asymp
(qt)^{\frac12-\sigma}|L(1-s,\overline{\chi})|
\ee
follows directly from \eqref{eq:L3_hist};
see \cite[Cors.\,10.5 and 10.10]{MontVau}.
Combining \eqref{eq:L5_hist} and the standard bound
(cf.~\cite[Cor.\,1.17 and Lem.\,10.15]{MontVau})
\[
L(s,\chi)\ll
(1+qt)^{1-\sigma}\log qt\qquad(\sigma\ge\tfrac12),
\]
one infers the classical convexity bound
\be\label{eq:L6_hist}
L(s,\chi)
\ll
\log qt
\begin{cases}
1&\quad\hbox{if $\sigma\ge 1$},\\
(qt)^{(1-\sigma)/2}
&\quad\hbox{if $0\le\sigma\le 1$},\\
(qt)^{1/2-\sigma}
&\quad\hbox{if $\sigma\le 0$};
\end{cases}
\ee
see, e.g., \cite[Thm.\,5.53]{IwanKowal}.
The convexity principle originates in the work of
Hadamard~\cite{Hadamard1896}
and Phragmén--Lindelöf~\cite{PhragmenLindelof1908}.
An early application to $\zeta(s)$ was given by
Hardy and Littlewood~\cite{HardyLittlewood1918}.

Finally, \eqref{eq:L3_hist} implies a functional equation
for the logarithmic derivative
$\frac{L'}{L}(s,\chi)$ as in~\cite[Eq.\,(10.35)]{MontVau}, namely
\be\label{eq:L7_hist}
\frac{L'}{L}(s,\chi)
=-\frac{L'}{L}(1-s,\overline\chi)-\log\frac{q }{2\pi}
-\frac{\Gamma'}{\Gamma}(1-s)
+\tfrac{1}{2}\pi\cot\tfrac{1}{2}\pi(s+\kappa).
\ee
This identity is what allows us, on the right edge of the contour,
to trade the factor $\frac{L'}{L}(1-s,\ochrone)$
appearing in $K(s)$ for $-\frac{L'}{L}(s,\chrone)$ plus
explicit lower-order terms. These explicit terms are responsible
for the main term in Theorem~\ref{thm:BLmain}.

\subsection{Bounds on certain sums}\label{sec:bounds}

\begin{lemma}\label{lem:Gkappa}
Let $x\ge 10$ and $\nu\ge 0$ be integers.
Taking $c\defeq 1+\tfrac{1}{\log x}$, we have
\[
\ssum{n\ge 2\\n\ne x}
\frac{(\log n)^\nu}{n^c|\log(x/n)|}\llsym{\nu}(\log x)^{\nu+1}.
\]
\end{lemma}

\begin{proof}
The denominator $|\log(x/n)|$ is bounded away from zero except
when $n$ is close to~$x$, so the sum is controlled
by the integers near $x$; away from $x$ the series
converges like $\zeta(c)$ and its derivatives. We make this precise
by splitting the sum according to the size of $n$ relative to $x$.
When $n<\frac12x$ or $n>\frac32x$, the lower bound
$|\log(x/n)|\gg 1$ holds, hence the contribution
from all $n$ in these ranges is
\[
\ll \sum_{n=2}^{\infty}\frac{(\log n)^\nu}{n^{c}}
=(-1)^\nu\zeta^{(\nu)}(c)\ll \frac{1}{(c-1)^{\nu+1}}
=(\log x)^{\nu+1},
\]
where we used $c-1=1/\log x$ together with the standard bound
$\zeta^{(\nu)}(c)\ll(c-1)^{-(\nu+1)}$ as $c\to 1^+$.
Next, consider integers $n$ in the range $\frac12x\le n<x$.
Here, $n\asymp x$ so we have $n^c\asymp x^c=\er{} x\asymp x$.
For the logarithm, the elementary inequality
$\log(1+u)\ge u/(1+u)$ applied with $u=(x-n)/n$ gives
\[
\log(x/n)\ge\frac{x-n}{x}.
\]
Hence, the contribution from this interval is
\[
\ll\ssum{\frac12x\le n<x}\frac{(\log n)^\nu}{x-n}
\le(\log x)^\nu\ssum{m\le\frac12x}\frac{1}{m}\ll(\log x)^{\nu+1}.
\]
Similarly, for integers $n$ with $x<n\le\frac32x$, we have
\[
|\log(x/n)|=\log(n/x)\ge\frac{n-x}{2x}
\mand
n^c\asymp x,
\]
so the corresponding contribution is
\[
\ll\ssum{x<n\le\frac32x}\frac{(\log n)^\nu}{n-x}
\le(\log \tfrac32x)^\nu\ssum{m\le\frac12x}\frac{1}{m}
\llsym{\nu}(\log x)^{\nu+1}.
\]
The result follows.
\end{proof}

\bigskip

Next, recall that there is an absolute constant $c>0$ such that
\be\label{eq:delaVP1}
\psi(x)\defeq\sum_{n\le x}\Lambda(n)
=x+O\big(x\,\er^{-c\sqrt{\log x}}\big).
\ee
More generally, for any character $\chi$ of modulus
$q\le \er^{c\sqrt{\log x}}$, the twisted Chebyshev function
satisfies
\be\label{eq:delaVP2}
\sum_{n\le x}\chi(n)\Lambda(n)
=\ind{\chi=\chi_{0,q}}\cdot x+O\big(x\,\er^{-c\sqrt{\log x}}\big),
\ee
where $\chi_{0,q}$ denotes the principal character. Furthermore,
if $\chi$ is nonprincipal, then
\be\label{eq:delaVP3}
\sum_{n\le x}\frac{\chi(n)}{n}\llsym{q} 1
\mand
\sum_{n\le x}\frac{\chi(n)\Lambda(n)}{n}\llsym{q} 1.
\ee
These estimates are classical consequences of the de la Vall\'ee Poussin
zero-free region for $\zeta(s)$ and Dirichlet $L$-functions, combined with
the explicit formula and partial summation; we refer the reader to the
excellent treatments in Davenport~\cite{Davenport},
Montgomery and Vaughan~\cite{MontVau},
Iwaniec and Kowalski~\cite{IwanKowal}, and Koukoulopoulos~\cite{Koukou}.
We caution the reader that the implied constants in~\eqref{eq:delaVP2}
and~\eqref{eq:delaVP3} are ineffective whenever $L(s,\chi)$
possesses a Siegel zero, owing to the standard ineffectivity in
Siegel's theorem.

\begin{remark}\label{chi0fails}
If $\chi$ is principal, then both sums in \eqref{eq:delaVP3} are
$O_q(\log x)$; we use this in Remark~\ref{chi0fails2} below.
\end{remark}

The next lemma plays a crucial role in our proof of
Theorem~\ref{thm:BLmain} in \S\ref{sec:BLmain}.
It asserts that, under mild hypotheses on the pair
$(\chrone,\chrtwo)$, the sums over a short interval $(t-\delta,t]$
of the arithmetical function $k=(\chrone\Lambda)*\chrtwo$ and of
certain twists of $k$ are $O_Q(\delta)$, a logarithm smaller than
the trivial bound $O_Q(\delta\log t)$.
This estimate is the key input on the left edge of the contour,
where it reduces the integral to an admissible error.

\begin{lemma}\label{lem:crucial}
Fix $Q\ge 1$, and let $\chrone,\chrtwo$ be Dirichlet characters modulo
$\qone$ and $\qtwo$, respectively, with $\qone,\qtwo\le Q$.
Define the arithmetical function $k=k(\chrone,\chrtwo)$ by
\be\label{eq:kdef}
k(n)\defeq\big\{(\chrone\cdot\Lambda)*\chrtwo\big\}(n)
=\ssum{a,b\in\N\\ab=n}\chrone(a)\Lambda(a)\chrtwo(b)\qquad(n\in\N).
\ee
If $\chrone,\chrtwo$ are not both principal, then there exist an
absolute constant $\widetilde C>0$ and a threshold $t_0=t_0(Q)$ such
that
\begin{equation}\label{eq:kk1}
\sum_{t-\delta<n\le t}k(n)\llsym{Q}\delta
\end{equation}
for all $t\ge t_0$ and all $\delta$ satisfying
$t\,\er^{-\widetilde C\sqrt{\log t}}\le\delta\le t$.
Further, if $\chrone,\chrtwo$ are not both induced by the same
primitive character, then for every Dirichlet character $\chi$ modulo
$q\le Q$ and every $h\in\Z$, one has
\begin{align}
\label{eq:kk2}
\sum_{t-\delta<n\le t}k(n)\chi(n)&\llsym{Q}\delta,
\\
\label{eq:kk3}
\sum_{t-\delta<n\le t}k(n)\e(-nh/q)&\llsym{Q}\delta,
\end{align}
in the same range of $t$ and $\delta$.
\end{lemma}

We first remark on the hypothesis governing \eqref{eq:kk2} and
\eqref{eq:kk3}. Since equal characters are induced by the same
primitive character, the hypothesis implies $\chrone\ne\chrtwo$;
when $\chrone$ and $\chrtwo$ are themselves primitive, it is
equivalent to the condition $\chrone\ne\chrtwo$. The hypothesis
cannot be dropped, for if $\chrone=\chrtwo$, then the bounds
\eqref{eq:kk2} and \eqref{eq:kk3} weaken to $O_Q(\delta\log t)$;
see Remark~\ref{chi0fails2} below.

The condition $\chrone\ne\chrtwo$ is not needed for \eqref{eq:kk1}.
Indeed, if $\chrone=\chrtwo$, then $\chrone$ is nonprincipal by
hypothesis, and $k(n)=\chrone(n)\log n$, so partial summation yields
the stronger estimate
\[
\sum_{t-\delta<n\le t}k(n)\llsym{\qone}\log t\ll\delta.
\]
By contrast, the hypothesis that $\chrone,\chrtwo$ are not both
principal cannot be dropped. If both are principal, then $k(n)$
agrees with $\log n$ up to the coprimality conditions $(a,\qone)=1$
and $(b,\qtwo)=1$ in the convolution, and the sum in \eqref{eq:kk1}
has genuine order $\delta\log t$; the coprimality conditions reduce
the implied constant by factors of $\phi(\qone)/\qone$ and
$\phi(\qtwo)/\qtwo$ but do not reduce the order.
Remark~\ref{chi0fails2} below, placed within the proof, records
the precise point at which this hypothesis enters the argument.

\begin{proof}
We begin with \eqref{eq:kk1}. Fix an absolute constant
$\widetilde C$ satisfying $0<\widetilde C<c/\sqrt{2}$, where $c$ is
the constant in \eqref{eq:delaVP1}. Since
\[
t(\log t)\,\er^{-c\sqrt{\frac12\log t}}\le t\,\er^{-\widetilde C\sqrt{\log t}}
\]
for all sufficiently large $t$, it suffices, after enlarging $t_0$
if necessary, to establish \eqref{eq:kk1} for all $\delta$ in the
wider range
\be\label{eq:range}
t(\log t)\,\er^{-c\sqrt{\frac12\log t}}\le\delta\le t.
\ee

Using Dirichlet's hyperbola method, we write
\[
\sum_{t-\delta<n\le t} k(n)
=\ssum{a,b\in\N\\ t-\delta<ab\le t}\chrone(a)\Lambda(a)\chrtwo(b)
=S_a+S_b-S_{ab},
\]
where
\dalign{
S_a&\defeq\sum_{a\le\sqrt{t}}\chrone(a)\Lambda(a)
\sum_{(t-\delta)/a<b\le t/a}\chrtwo(b),\\
S_b&\defeq\sum_{b\le\sqrt{t}}\chrtwo(b)
\sum_{(t-\delta)/b<a\le t/b}\chrone(a)\Lambda(a),\\
S_{ab}&\defeq\ssum{a,b\le\sqrt{t}\\t-\delta<ab\le t}
\chrone(a)\Lambda(a)\chrtwo(b).
}
Every pair $(a,b)$ occurring in $S_{ab}$ satisfies
$(t-\delta)/\sqrt{t}<a\le\sqrt{t}$, and therefore
\[
S_{ab}\ll\sum_{b\le\sqrt{t}}\,\sum_{(t-\delta)/\sqrt{t}<a\le\sqrt{t}}\Lambda(a)
\le\sqrt{t}\,\Big\{\psi\big(\sqrt{t}\,\big)
-\psi\Big(\frac{t-\delta}{\sqrt{t}}\Big)\Big\}
=\delta+O\big(t\,\er^{-c\sqrt{\frac12\log t}}\big)\ll\delta,
\]
where \eqref{eq:delaVP1} yields the penultimate step and
\eqref{eq:range} the last step.

Next, we estimate $S_a$. The argument splits according to whether or
not $\chrtwo$ is principal.

Suppose first that $\chrtwo=\chi_{0,\qtwo}$, and so $\chrone$ is
nonprincipal by hypothesis. Counting the integers in $(y,z]$ that
are coprime to $\qtwo$, we see that
\[
\sum_{y<b\le z}\chi_{0,\qtwo}(b)
=\frac{\phi(\qtwo)}{\qtwo}(z-y)+O_Q(1)\qquad(0<y<z).
\]
Taking $y=(t-\delta)/a$ and $z=t/a$, we get
\dalign{
S_a&=\sum_{a\le\sqrt{t}}\chrone(a)\Lambda(a)
\Big\{\frac{\phi(\qtwo)}{\qtwo}\frac{\delta}{a}+O_Q(1)\Big\}\\
&=\frac{\phi(\qtwo)}{\qtwo}\,\delta\sum_{a\le\sqrt t}
\frac{\chrone(a)\Lambda(a)}{a}+O_Q\Big(\sum_{a\le\sqrt t}\Lambda(a)\Big)
\llsym{Q}\delta+\sqrt{t}\ll\delta,
}
where we used \eqref{eq:delaVP3} with the nonprincipal character
$\chi\defeq\chrone$ together with the bound
$\sum_{a\le\sqrt t}\Lambda(a)\ll\sqrt t\ll\delta$.

Suppose now that $\chrtwo$ is nonprincipal, with $\chrone$ arbitrary.
The sum of $\chrtwo$ over any interval of consecutive integers is
bounded by $\qtwo$, so the inner sum over $b$ is $O_Q(1)$. Since
$|\chrone(a)|\le 1$, it follows that
\[
S_a\llsym{Q}\sum_{a\le\sqrt{t}}\Lambda(a)\ll\sqrt{t}\ll\delta.
\]
Thus $S_a=O_Q(\delta)$ in both cases.

\begin{remark}\label{chi0fails2}
Here, we record where the hypothesis that $\chrone,\chrtwo$ are not
both principal enters the proof. If $\chrone$ and $\chrtwo$ were
both principal, the argument above would apply \eqref{eq:delaVP3}
with a principal character. By Remark~\ref{chi0fails}, the sum over
$a$ in the second line of the display would then be of size
$O_Q(\log t)$ rather than $O_Q(1)$, and the bound on $S_a$ would
weaken to $O_Q(\delta\log t)$. The same applies to $S_b$ below.
\end{remark}

Finally, we estimate $S_b$. As before, the argument splits
according to whether or not $\chrone$ is principal.
In both cases, we evaluate the inner sum
by applying \eqref{eq:delaVP2} with $x\defeq t/b$, noting that
$\log(t/b)\ge\tfrac12\log t$ whenever $b\le\sqrt{t}$.

Suppose first that $\chrone=\chi_{0,\qone}$. Then $\chrtwo$ is
nonprincipal by hypothesis, and
\dalign{
S_b&=\sum_{b\le\sqrt{t}}\chrtwo(b)
\sum_{(t-\delta)/b<a\le t/b}\chi_{0,\qone}(a)\Lambda(a)\\
&=\sum_{b\le\sqrt{t}}\chrtwo(b)\Big\{
\frac{\delta}{b}+O\Big(\frac{t}{b}\,\er^{-c\sqrt{\frac12\log t}}\Big)\Big\}
\llsym{Q}\delta+t(\log t)\,\er^{-c\sqrt{\frac12\log t}}\ll\delta,
}
where we used the first bound in \eqref{eq:delaVP3} with the
nonprincipal character $\chi\defeq\chrtwo$, and \eqref{eq:range} in
the last step.

If $\chrone$ is nonprincipal, then the main term in
\eqref{eq:delaVP2} vanishes, hence the inner sum satisfies
\[
\sum_{(t-\delta)/b<a\le t/b}\chrone(a)\Lambda(a)
\ll\frac{t}{b}\,\er^{-c\sqrt{\frac12\log t}}.
\]
Since $|\chrtwo(b)|\le 1$, it follows from \eqref{eq:range} that
\[
S_b\ll\sum_{b\le\sqrt t}\frac{t}{b}\,\er^{-c\sqrt{\frac12\log t}}
\ll t(\log t)\,\er^{-c\sqrt{\frac12\log t}}\ll\delta.
\]
Thus $S_b=O_Q(\delta)$ in both cases.

Combining the estimates for $S_{ab}$, $S_a$, and $S_b$ completes
the proof of \eqref{eq:kk1}.

\bigskip

To prove \eqref{eq:kk2}, fix a Dirichlet character $\chi$ modulo
$q\le Q$ and observe that
\[
k(n)\chi(n)=\big\{(\chi\chrone\cdot\Lambda)*(\chi\chrtwo)\big\}(n)
=k(\chi\chrone,\chi\chrtwo;n)\qquad(n\in\N),
\]
since $\chi$ is completely multiplicative. Consequently, it suffices
to show that the pair $(\chi\chrone,\chi\chrtwo)$ satisfies the
hypotheses under which \eqref{eq:kk1} was established, with $Q^2$ in
place of $Q$. Since $Q$ is fixed, the resulting bound
$O_{Q^2}(\delta)$ and threshold $t_0(Q^2)$ remain admissible.

The moduli of $\chi\chrone$ and $\chi\chrtwo$ divide $q\qone$ and
$q\qtwo$, respectively, and thus do not exceed $Q^2$. It remains to
verify that $\chi\chrone$ and $\chi\chrtwo$ are not both principal.
Note that the hypothesis of the first assertion does not survive the
twist, as $\chi\chrone$ can be principal even when $\chrone$ is not
(take $\chi=\overline{\chrone}$ with $q=\qone$); this is why the
second assertion of the lemma carries the stronger hypothesis.
Suppose, then, that $\chi\chrone$ and $\chi\chrtwo$ are both
principal. For every integer $n$ coprime to $q\qone\qtwo$ we have
$\chi(n)\chrone(n)=1=\chi(n)\chrtwo(n)$, and since $\chi(n)\ne 0$,
it follows that $\chrone(n)=\chrtwo(n)$ for all such $n$. Two
characters that agree on all integers coprime to a common multiple
of their moduli are induced by the same primitive character, so
$\chrone$ and $\chrtwo$ would be induced by the same primitive
character, contrary to hypothesis. This establishes \eqref{eq:kk2}.

\bigskip

Finally, we establish \eqref{eq:kk3}. Fix $h\in\Z$. Since
$\chrtwo(b)=0$ unless $\gcd(b,\qtwo)=1$, every nonzero summand in
$k(n)=\sum_{ab=n}\chrone(a)\Lambda(a)\chrtwo(b)$ has $b$ coprime to
$\qtwo$ and $a$ a prime power $p^{\,j}$. Hence a nonzero term
$k(n)$ with $\gcd(n,\qtwo)>1$ arises only from $n=p^{\,j}b$ with
$p\mid\qtwo$, and the contribution of these terms satisfies
\dalign{
\ssum{t-\delta<n\le t\\\gcd(n,\qtwo)>1}|k(n)|
&\le\sum_{p\mid\qtwo}\log p\sum_{1\le j\le\frac{\log t}{\log p}}
\ssum{(t-\delta)/p^j<b\le t/p^j\\\gcd(b,\qtwo)=1}1\\
&\ll\sum_{p\mid\qtwo}\log p\sum_{1\le j\le\frac{\log t}{\log p}}
\Big\{\frac{\delta}{p^{\,j}}+1\Big\}
\llsym{Q}\delta+\log t\ll\delta,
}
where we used \eqref{eq:range} in the last step. Sorting the
remaining terms into residue classes modulo $\qtwo$, we obtain
\dalign{
\sum_{t-\delta<n\le t}k(n)\e(-nh/\qtwo)
&=\ssum{t-\delta<n\le t\\\gcd(n,\qtwo)=1}k(n)\e(-nh/\qtwo)+O_Q(\delta)\\
&=\sum_{m\in(\Z/\qtwo\Z)^\times}\e(-mh/\qtwo)
\ssum{t-\delta<n\le t\\n\equiv m\bmod\qtwo}k(n)+O_Q(\delta).
}
For every $m\in(\Z/\qtwo\Z)^\times$, orthogonality gives
\[
\ind{n\equiv m\bmod\qtwo}
=\frac{1}{\phi(\qtwo)}\sum_{\chi\bmod\qtwo}\chi(n)\overline\chi(m)
\qquad(n\in\Z),
\]
and therefore
\[
\sum_{t-\delta<n\le t}k(n)\e(-nh/\qtwo)
=\frac{1}{\phi(\qtwo)}\sum_{\chi\bmod\qtwo}
\bigg\{\sum_{m\bmod\qtwo}\overline\chi(m)\e(-mh/\qtwo)\bigg\}
\sum_{t-\delta<n\le t}k(n)\chi(n)+O_Q(\delta).
\]
The Gauss sum over $m$ is trivially $O_Q(1)$, and \eqref{eq:kk2}
bounds the sum over $n$ by $O_Q(\delta)$ for each of the
$\phi(\qtwo)\le Q$ characters $\chi$; hence the entire expression
is $O_Q(\delta)$. This establishes \eqref{eq:kk3} and completes the
proof of the lemma.
\end{proof}

\subsection{The function $\cM_\chi(s)$}\label{sec:M}

Stirling's formula for the gamma function leads to the estimate
\[
\cM_\chi(1-c-it)={\mathfrak g}(\chi)q^{c-1}\er^{-\pi i/4}
\exp\Big(it\log\Big(\frac{qt}{2\pi\er}\Big)\Big)
\Big(\frac{t}{2\pi}\Big)^{c-1/2}\big\{1+O(t^{-1})\big\},
\]
which is uniform for $c\in[\frac{1}{10},2]$ and $t\ge 1$;
see \cite[Lem.\,2.1]{Banks}.
Moreover, using the method of stationary phase,
Gonek~\cite[Lem.\,2]{Gonek} showed that
\[
\int_a^b\exp\Big(it\log\Big(\frac{t}{u\er}\Big)\Big)
\Big(\frac{t}{2\pi}\Big)^{c-1/2}\dd t
=(2\pi)^{1-c}u^c\er^{-iu+\pi i/4}\cdot\ind{(a,b]}(u)+O(E)
\]
holds uniformly for $c\in[\frac{1}{10},2]$ and $10<a<b\le 2a$,
where the error term $E$ is bounded explicitly in terms of
$a$, $b$, $c$, and $u$.
Combining these results, we arrive at the following variant of a
lemma of Conrey, Ghosh, and Gonek \cite[Lem.\,1]{ConGhoGon}.
It asserts that a contour integral of $\cM_\chi(1-s)v^{-s}$ along a
short vertical segment behaves like a sharp cutoff. The integral is
essentially $\mathfrak g(\chi)q^{-1}\e(-v/q)$ when $2\pi v$ lies in
the interval $(qT_1,qT_2]$ and is negligible otherwise, with the
error term $E$ controlling the transition near the endpoints.

\begin{lemma}\label{lem:ConGhoGon}
Uniformly for $v>0$, $c\in[\frac{1}{10},2]$, and
$10<T_1<T_2\le 2T_1$, we have
\[
\frac{1}{2\pi i}\int_{c+iT_1}^{c+iT_2}\cM_\chi(1-s)v^{-s}\dd s
={\mathfrak g}(\chi)q^{-1}\e(-v/q)\cdot\ind{(qT_1,qT_2]}(2\pi v)+O(E),
\]
where
\[
E\defeq\frac{q^{c-1/2}}{v^c}
\bigg(T^{c-1/2}+\frac{T^{c+1/2}}{|T-2\pi v/q|+T^{1/2}}+1\bigg)
\qquad(T\defeq T_2).
\]
\end{lemma}

{\Large\section{Proof of Theorem~\ref{thm:BLmain}}\label{sec:BLmain}}

For convenience, we restate the assertion \eqref{eq:BLest} of
Theorem~\ref{thm:BLmain}. For every $\Delta$ satisfying
$T\,\er^{-C\sqrt{\log T}}\le\Delta\le T/\log T$, we have
\[
\ssum{\rho=\beta+i\gamma\\T-\Delta<\gamma\le T}
x^\rho L(\rho,\chrtwo)
=\frac{\chrtwo(x)}{2\pi}\,\Delta\log T
+O(\Delta).
\]

\begin{proof}
Fix a positive constant $C < \widetilde C$, where $\widetilde C$ is the
constant in Lemma~\ref{lem:crucial}. Recalling the convention
given in \S\ref{sec:notation}, we continue to work in
the region $\cR\defeq\{s\in\C : t\ge 10\}$.

We fix notation that will be used throughout the proof, writing
\[
\Lone\defeq\log\frac{\qone T}{2\pi}
\mand
\Ltwo\defeq\log\frac{\qtwo T}{2\pi},
\]
and defining
\[
 T_2\defeq T\mand T_1\defeq T-\Delta,
\]
so that the sum in \eqref{eq:BLest} takes the form
\[
\ssum{\rho=\beta+i\gamma\\T_1<\gamma\le T_2}
x^\rho L(\rho,\chrtwo).
\]

In view of \eqref{eq:L2_hist}, we may choose ordinates
$T_j^\star\in[T_j,T_j+1]$ for $j\in\{1,2\}$ at which $L'/L$ satisfies
\be\label{eq:L2_bd2}
\frac{L'}{L}(\sigma\pm iT_j^\star,\chrone)
\ll \Lone^2\qquad (\sigma\in[-1,2],~j\in\{1,2\}).
\ee
The bound \eqref{eq:L2_bd2} is available only at these special ordinates, not
at the prescribed $T_j$. We therefore run the contour argument below with
$T_1^\star,T_2^\star$ in place of $T_1,T_2$, proving the asymptotic for the
shifted window $T_1^\star<\gamma\le T_2^\star$ of width
$\Delta^\star\defeq T_2^\star-T_1^\star=\Delta+O(1)$, and then we 
correct back to the original window in two steps.

First, the sum on the left-hand side of \eqref{eq:BLest} differs
from the corresponding sum with $T_1^\star$ and $T_2^\star$ by
\[
\ssum{\rho=\beta+i\gamma\\T_1<\gamma\le T_1^\star}
x^\rho L(\rho,\chrtwo)
-\ssum{\rho=\beta+i\gamma\\T_2<\gamma\le T_2^\star}
x^\rho L(\rho,\chrtwo).
\]
Since both sums run over at most $O(\Lone)$ zeros
(see \eqref{eq:L1_hist}), it follows from \eqref{eq:L6_hist} that
\[
\ssum{\rho=\beta+i\gamma\\T_j<\gamma\le T_j^\star}
x^\rho L(\rho,\chrtwo)
\ll\Lone\max_{\substack{0\le\sigma\le 1\\T_j<t\le T_j^\star}}
\big\{x^\sigma\big|L(\sigma+it,\chrtwo)\big|\big\}
\ll\big(x+(\qtwo T)^{1/2}\big)\Lone\Ltwo.
\]
This bound is admissible, as the last term
is $O_{Q,x}(T^{1/2}(\log T)^2)$,
whereas $\Delta\ge T^{1/2}(\log T)^2$ for all large $T$.
Second, replacing $\Delta^\star$ by $\Delta$ in the main term
$\frac{\chrtwo(x)}{2\pi}\Delta^\star\log T$
produced below changes it by
\[
\frac{\chrtwo(x)}{2\pi}\,(\Delta^\star-\Delta)\log T=O(\log T)
\llsym{Q,x}\Delta.
\]
Both corrections can therefore be absorbed into the error term of
\eqref{eq:BLest}, and we may proceed with \eqref{eq:L2_bd2}.

For the remainder of the proof we work with the shifted window, writing
$T_1,T_2$ for $T_1^\star,T_2^\star$ and $\Delta$ for $\Delta^\star$. 
The two corrections above show that the resulting asymptotic
for the shifted window
implies \eqref{eq:BLest} for the original one.

The function
\be\label{eq:Kdefn}
K(s)
\defeq -\frac{L'}{L}(1-s,\ochrone)\cdot L(s,\chrtwo)
\ee
has a meromorphic continuation to the entire complex plane.
Its poles in the region $\cR$ are simple and located at those zeros
$\rho$ of $L(s,\chrone)$ that are \emph{not}
zeros of $L(s,\chrtwo)$. The
residue of $K(s)x^s$ at such a $\rho$ is
$m_\rho\,x^\rho L(\rho,\chrtwo)$,
where $m_\rho$ denotes the multiplicity of $\rho$,
so summing over $\rho$ by
Cauchy's theorem reproduces the left-hand side of \eqref{eq:BLest},
in which each zero is counted with its multiplicity.
In particular, when $\chrtwo=\chrone$ the function $K$ has no poles in
$\cR$ at all, since every zero $\rho$ of $L(s,\chrone)$ is then
also a zero of $L(s,\chrtwo)$.
In that case the asymptotic \eqref{eq:BLest} would assert
a nonzero main term for a sum that is identically zero. This is why the
hypothesis $\chrone\ne\chrtwo$ cannot be dropped.

From now on, we write
\[
c\defeq 1+\frac{1}{\log x}
\]
and denote by $\cC$ the oriented rectangle depicted in Figure~\ref{fig:contour}, namely
\[
\cC:\quad
c+iT_1
~~\longrightarrow~~c+iT_2
~~\longrightarrow~~1-c+iT_2
~~\longrightarrow~~1-c+iT_1
~~\longrightarrow~~c+iT_1.
\]

\begin{figure}[H]
    \centering
    \includegraphics[width=4.75in]{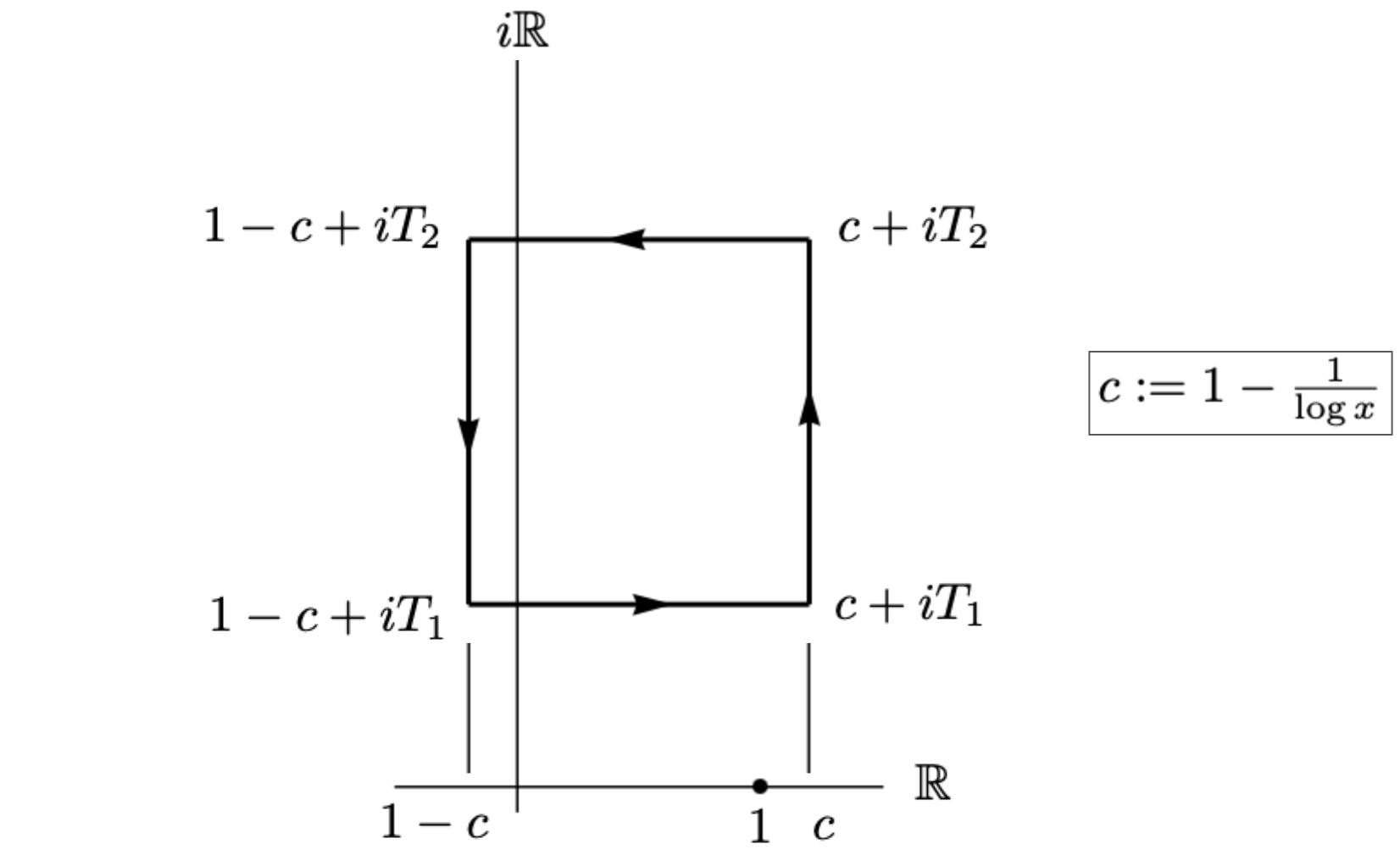}
    \caption{Contour of integration $\cC$.}
    \label{fig:contour}
\end{figure}

By Cauchy's theorem,
\dalign{
\ssum{\rho=\beta+i\gamma\\T_1<\gamma\le T_2}x^\rho L(\rho,\chrtwo)
&=\frac{1}{2\pi i}\oint_\cC K(s) x^s\dd s\\
&=\frac{1}{2\pi i}\bigg(\int_{c+iT_1}^{c+iT_2}
+\int_{c+iT_2}^{1-c+iT_2}
+\int_{1-c+iT_2}^{1-c+iT_1}
+\int_{1-c+iT_1}^{c+iT_1}\bigg)
K(s) x^s\dd s\\
&=I_1+I_2+I_3+I_4\quad\text{(say)}.
}

We estimate each term $I_j$ individually.

\medskip

\noindent{\bf Estimate for $I_1$}.
Using \eqref{eq:L7_hist} and standard approximations for
the digamma, cotangent, and logarithm functions, we
obtain the uniform estimate
\[
\frac{L'}{L}(1-s,\ochrone)
=-\frac{L'}{L}(s,\chrone)-\log\frac{\qone}{2\pi}
-\log t+O(t^{-1})\qquad(t\ge 10),
\]
which implies that
\[
\frac{L'}{L}(1-s,\ochrone)
=-\frac{L'}{L}(s,\chrone)-\Lone+O(\Delta T^{-1})
\qquad(T_1\le t\le T_2).
\]
Thus, invoking \eqref{eq:Kdefn}, we find that
\dalign{
I_1&=\frac{-1}{2\pi}\int_{T_1}^{T_2}
\frac{L'}{L}(1-c-it,\ochrone)  L(c+it,\chrtwo) x^{c+it}\dd t\\
&=\frac{1}{2\pi}\int_{T_1}^{T_2}
\bigg\{\frac{L'}{L}(c+it,\chrone)
+\Lone+O(\Delta T^{-1})\bigg\}
L(c+it,\chrtwo) x^{c+it}\dd t\\
&=I_{1a}+I_{1b}+O(I_{1c})\quad\text{(say)}.
}
Using the Dirichlet series representation
\be\label{eq:Dirseries}
\frac{L'}{L}(s,\chrone)\cdot L(s,\chrtwo)
=-\sum_{n\in\N}\frac{k(n)}{n^s}\qquad(\sigma>1),
\ee
where $k$ is the arithmetical function introduced in \eqref{eq:kdef},
\[
k(n)=\big\{(\chrone\cdot\Lambda)*\chrtwo\big\}(n)
=\ssum{a,b\in\N\\ab=n}\chrone(a)\Lambda(a)\chrtwo(b),
\]
we obtain the estimate
\dalign{
I_{1a}&=\frac{1}{2\pi}\int_{T_1}^{T_2}
\frac{L'}{L}(c+it,\chrone)
L(c+it,\chrtwo) x^{c+it}\dd t
=-\frac{x^c}{2\pi}\sum_{n\ge 2}\frac{k(n)}{n^c}
\int_{T_1}^{T_2}\Big(\frac{x}{n}\Big)^{it}\dd t\\
&=-\frac{\Delta}{2\pi} k(x)
+O\bigg(x^c\ssum{n\ge 2\\n\ne x}
\frac{\log n}{n^c|\log(x/n)|}\bigg).
}
Here, the term $n=x$ contributes $\int_{T_1}^{T_2}\dd t=\Delta$ and
produces the first term on the right, whereas for each $n\ne x$,
\[
\int_{T_1}^{T_2}\Big(\frac{x}{n}\Big)^{it}\dd t
=\frac{(x/n)^{iT_2}-(x/n)^{iT_1}}{i\log(x/n)}\ll\frac{1}{|\log(x/n)|}.
\]
The triangle inequality gives
\[
\big|k(n)\big|\le \ssum{a,b\in\N\\ab=n}\Lambda(a)=\log n.
\]
Applying Lemma~\ref{lem:Gkappa}, and noting that
$x^c=\er x\asymp x$, we find that
\[
I_{1a}=-\frac{\Delta}{2\pi} k(x)+O(x(\log x)^2)
\llsym{Q,x}\Delta.
\]
By the same evaluation of the inner integral, we have
\dalign{
I_{1b}&=\frac{\Lone}{2\pi}\int_{T_1}^{T_2}
L(c+it,\chrtwo) x^{c+it}\dd t
=\frac{x^c\Lone}{2\pi}
\sum_{n\ge 1}\frac{\chrtwo(n)}{n^c}
\int_{T_1}^{T_2}\Big(\frac{x}{n}\Big)^{it}\dd t\\
&=\frac{\Delta\Lone}{2\pi}\chrtwo(x)
+O\bigg(x\Lone\ssum{n\ge 1\\n\ne x}
\frac{1}{n^c|\log(x/n)|}\bigg)\\
&=\frac{\Delta\Lone}{2\pi}\chrtwo(x)+O(x\log x\cdot\Lone),
}
where the final step follows from Lemma~\ref{lem:Gkappa} applied to
the terms with $n\ge 2$, the term $n=1$
contributing only $(\log x)^{-1}$. Consequently,
\[
I_{1b}=\frac{\chrtwo(x)}{2\pi}\,\Delta\log T+O_{Q,x}(\Delta).
\]
Finally, \eqref{eq:L6_hist} gives
\[
I_{1c}\ll \frac{\Delta}{T}\,x\int_{T_1}^{T_2}
\big|L(c+it,\chrtwo)\big|\dd t
\ll \frac{\Delta}{T}\,x\int_{T_1}^{T_2}\log\qtwo t\dd t
\ll\frac{\Delta^2 x\Ltwo}{T},
\]
and since $\Delta\le T/\log T$ by hypothesis, we conclude that
\[
I_{1c}\llsym{Q,x}\Delta.
\]
Putting everything together, we have shown that
\be\label{eq:I1result}
I_1=\frac{\chrtwo(x)}{2\pi}\,\Delta\log T+O_{Q,x}(\Delta).
\ee

\bigskip\noindent{\bf Bounds for $I_2$ and $I_4$}.
In view of \eqref{eq:L2_bd2} for $T_2$, we have
\dalign{
I_2&=\frac{1}{2\pi i}\int_{1-c}^c
\frac{L'}{L}(1-\sigma-iT_2,\ochrone)
L(\sigma+iT_2,\chrtwo) x^{\sigma+iT_2} \dd\sigma\\
&\ll\Lone^2\int_{1-c}^c
\big|L(\sigma+iT_2,\chrtwo)\big| x^\sigma\dd\sigma.
}
Applying \eqref{eq:L6_hist} to the second integral, it follows that
\[
I_2\ll(x+(\qtwo T)^{1/2})\Lone^2\Ltwo\,(\qtwo T)^{c-1}.
\]
Since $x\ge 10$, we have
\[
(\qtwo T)^{c-1}=(\qtwo T)^{1/\log x}
\llsym{Q}T^{0.45},
\]
and therefore
\be\label{eq:I2result}
I_2\llsym{Q,x}T^{0.95}(\log T)^3\ll\Delta.
\ee
The edge $I_4$ is the segment at height $T_1$, traversed from $1-c+iT_1$ to
$c+iT_1$. The estimate is identical to that for $I_2$, now using
\eqref{eq:L2_bd2} with $T_1$ to bound $\frac{L'}{L}(1-\sigma-iT_1,\ochrone)$.
We conclude similarly that
\be\label{eq:I4result}
I_4\llsym{Q,x}\Delta.
\ee

\bigskip\noindent{\bf Bound for $I_3$}. Finally, we bound $I_3$,
working with its complex conjugate. Applying the functional equation
\eqref{eq:L3_hist} to $L(s,\chrtwo)$
and the change of variable $s\mapsto 1-s$, we find that
\[
\overline{I_3}=\frac{1}{2\pi i}\int_{c+iT_1}^{c+iT_2}
\frac{L'}{L}(s,\chrone)\cM_{\ochrtwo}(1-s)L(s,\chrtwo)x^{1-s}\dd s.
\]
Recalling \eqref{eq:Dirseries}, we get
\[
\overline{I_3}=\frac{-x}{2\pi i}\sum_{n\in\N}k(n)\int_{c+iT_1}^{c+iT_2}
\cM_{\ochrtwo}(1-s)(nx)^{-s}\dd s.
\]
Applying Lemma~\ref{lem:ConGhoGon} to each integral,
and noting that $\tfrac12<c<2$, we deduce that
\be\label{eq:angst0}
\overline{I_3}=M+O(E_1+E_2),
\ee
where
\[
M\defeq\frac{-x{\mathfrak g}(\ochrtwo)}{\qtwo}
\ssum{n\in\N\\\qtwo T_1<2\pi nx\le \qtwo T_2}k(n)\e(-nx/\qtwo),
\]
and, using $|k(n)|\le\log n$, the error terms are
\dalign{
E_1&\defeq\sum_{n\in\N}(\log n)\frac{(\qtwo T)^{c-1/2}}{(nx)^c},\\
E_2&\defeq\sum_{n\in\N}(\log n)\frac{(\qtwo T)^{c-1/2}}{(nx)^c}
\frac{T}{|T-2\pi nx/\qtwo|+T^{1/2}}.
}

First, we bound $M$. Setting
$\tau\defeq\frac{\qtwo T}{2\pi x}$ and $\delta\defeq\frac{\qtwo\Delta}{2\pi x}$,
we have
\[
M=-\frac{x\,{\mathfrak g}(\ochrtwo)}{\qtwo}
\sum_{\tau-\delta<n\le\tau}k(n)\e(-nx/\qtwo)
\llsym{Q}x\,\bigg|\sum_{\tau-\delta<n\le\tau}k(n)\e(-nx/\qtwo)\bigg|
\]
since the Gauss sum is $O_Q(1)$.
To bound the above sum, we invoke \eqref{eq:kk3} of Lemma~\ref{lem:crucial},
taking $h$ to be the integer $x$ of Theorem~\ref{thm:BLmain}. The hypothesis of
the lemma is met because $\chrone$ and $\chrtwo$ are distinct primitive characters;
this is the only point in the proof at which the assumption $\chrone\ne\chrtwo$ is
used. To apply \eqref{eq:kk3} we must also check that
$\tau\,\er^{-\widetilde C\sqrt{\log\tau}}\le\delta\le\tau$, i.e., that
\[
\frac{\qtwo T}{2\pi x}\,
\exp\!\Big\{-\widetilde C\sqrt{\log\tfrac{\qtwo T}{2\pi x}}\Big\}
\le\frac{\qtwo\Delta}{2\pi x}\le\frac{\qtwo T}{2\pi x}.
\]
Recalling that $C<\widetilde C$ and that
$T\,\er^{-C\sqrt{\log T}}\le\Delta\le T/\log T$, these inequalities are
easily verified for all $T\ge T_0(Q,x)$, as is the threshold
requirement $\tau\ge t_0(Q)$ of the lemma. Hence \eqref{eq:kk3}
applies and yields
\[
\sum_{\tau-\delta<n\le\tau}k(n)\e(-nx/\qtwo)
\llsym{Q}\delta=\frac{\qtwo\Delta}{2\pi x}\llsym{Q}\frac{\Delta}{x}.
\]
We conclude that
\be\label{eq:angst1}
M\llsym{Q}\Delta.
\ee

We pause to record what fails in the degenerate case
$\chrone=\chrtwo$. There the two characters are induced by the same
primitive character, so \eqref{eq:kk3} is no longer available; indeed,
by Remarks~\ref{chi0fails} and~\ref{chi0fails2}, the bound weakens to
$O_Q(\delta\log t)$, and $M$ acquires a genuine main term of size
$\asymp\Delta\log T$. This term would cancel the main term of
\eqref{eq:I1result} arising from $I_{1b}$, consistent with the fact
that the left-hand side of \eqref{eq:BLest} vanishes identically when
$\chrone=\chrtwo$.

Next, we bound the error terms of \eqref{eq:angst0}.
As in the estimate for $I_2$, the hypothesis $x\ge 10$ implies that
\[
(\qtwo T)^{c-1/2}=(\qtwo T)^{1/2}(\qtwo T)^{1/\log x}
\llsym{Q}T^{0.95},
\]
and therefore
\be\label{eq:dog}
E_1\asymp \frac{(\qtwo T)^{c-1/2}}{x}\big|\zeta'(c)\big|
\ll\frac{(\qtwo T)^{c-1/2}}{x}\,(\log x)^2
\llsym{Q,x}\Delta.
\ee
Also,
\be\label{eq:cat}
E_2\asymp \frac{(\qtwo T)^{c-1/2}}{x}\sum_{n\in\N}\frac{\log n}{n^c}
\frac{T}{|T-2\pi nx/\qtwo|+T^{1/2}}.
\ee
To estimate this last sum, we adapt the argument of \cite[pp.\,7--8]{Banks},
partitioning the integers $n\ge 2$ into the three disjoint sets
\dalign{
S_1&\defeq\{n\ge 2:|T-2\pi nx/\qtwo|>\tfrac12T\},\\
S_2&\defeq\{n\ge 2:|T-2\pi nx/\qtwo|\le T^{1/2}\},\\
S_3&\defeq\{n\ge 2:T^{1/2}<|T-2\pi nx/\qtwo|\le \tfrac12T\}.
}
First, the sum over $S_1$ satisfies
\[
\Sigma_1\defeq\sum_{n\in S_1}\frac{\log n}{n^c}\frac{T}{|T-2\pi nx/\qtwo|+T^{1/2}}
\ll\sum_{n\in S_1}\frac{\log n}{n^c}\le\big|\zeta'(c)\big|\ll(\log x)^2.
\]
Similarly,
\[
\Sigma_2\defeq\sum_{n\in S_2}\frac{\log n}{n^c}\frac{T}{|T-2\pi nx/\qtwo|+T^{1/2}}
\ll T^{1/2}\sum_{n\in S_2}\frac{\log n}{n^c}.
\]
For each $n\in S_2$, we have
\[
n=\frac{\qtwo T}{2\pi x}(1+O(T^{-1/2})),\qquad
\log n=\Ltwo-\log x+O(T^{-1/2})\ll\Ltwo,\qquad
n^c\ge n\gg\frac{\qtwo T}{x}.
\]
These estimates, combined with the bound
$|S_2|\ll\qtwo T^{1/2}/x+1$, yield
\[
\Sigma_2\ll\Ltwo+\frac{x\Ltwo}{\qtwo T^{1/2}}.
\]
It remains to treat the sum over $S_3$. We write
\[
\Sigma_3\defeq\sum_{n\in S_3}\frac{\log n}{n^c}\frac{T}{|T-2\pi nx/\qtwo|+T^{1/2}}
\le \sum_{\lfloor T^{1/2}\rfloor\le k\le \frac12 T}\Sigma_{3,k},
\]
where $\Sigma_{3,k}$ denotes the subsum over the set $S_{3,k}$ of
integers $n$ satisfying
\[
k\le |T-2\pi nx/\qtwo|\le k+1.
\]
For each $n\in S_{3,k}$, one has
\[
n=\frac{\qtwo T\pm\qtwo k}{2\pi x}+O(\qtwo/x),~\quad
\log n\ll\Ltwo,~\quad
n^c\ge n\gg\frac{\qtwo T}{x},~\quad
\frac{T}{|T-2\pi nx/\qtwo|+T^{1/2}}\ll\frac{T}{k}.
\]
Together with the bound $|S_{3,k}|\ll\qtwo/x+1$, these give
\[
\Sigma_{3,k}\ll\Big(\Ltwo+\frac{x\Ltwo}{\qtwo}\Big)\frac{1}{k}.
\]
Summing over $k$ gives
\[
\Sigma_3\ll\Big(\Ltwo+\frac{x\Ltwo}{\qtwo}\Big)\log T.
\]
Collecting the bounds on $\Sigma_1$, $\Sigma_2$,
and $\Sigma_3$ and applying \eqref{eq:cat}, we conclude that
\be\label{eq:angst2}
E_2\ll\frac{(\qtwo T)^{c-1/2}}{x}(\log x)^2
+\Big(\frac{1}{\qtwo}
+\frac{1}{x}\Big)(\qtwo T)^{c-1/2}\Ltwo\log T
\llsym{Q,x}\Delta.
\ee

Substituting \eqref{eq:angst1}, \eqref{eq:dog}, and \eqref{eq:angst2} into
\eqref{eq:angst0} yields
\be\label{eq:I3result}
I_3\llsym{Q,x}\Delta.
\ee
The theorem now follows by combining
\eqref{eq:I1result}, \eqref{eq:I2result}, \eqref{eq:I4result},
and \eqref{eq:I3result}.
\end{proof}

\end{document}